\RequirePackage{fix-cm}
\documentclass[onecolumn]{svjour3}          
\smartqed  
\usepackage{graphicx}

\usepackage[authoryear]{natbib}

\usepackage[T1]{fontenc}
\usepackage{times}
\usepackage{microtype}
\usepackage{color}
\usepackage{wrapfig}
\usepackage{amsfonts}
\usepackage{amssymb}
\usepackage{pictexwd}
\usepackage{fancyhdr}
\usepackage[all]{xy}
\usepackage{mathrsfs}
\usepackage{enumitem}

\newtheorem{teo}{Theorem}
\newtheorem{lema}{Lemma}
\newtheorem{cor}{Corolary}
\newtheorem{prop}{Proposition}
\newtheorem{defi}{Definition}
\newtheorem{exem}{Example}
\newtheorem{obs}{Remark}

\begin{document}

\title{Dynamic Ordered Weighted Averaging Functions for Complete Lattices\footnote{Preprint submited to Soft Computing}}


\author{Antonio Diego S. Farias \and Regivan H. N. Santiago \and Benjam\'in Bedregal
}


%

\institute{Antonio Diego S. Farias \at
	Federal Rural University of Semi-Arid - UFERSA, Universitary Campus of Pau dos Ferros, BR 226, KM 405, S\~ao Geraldo, 59.900-000, Pau dos Ferros, RN, Brazil\\
	\email{antonio.diego@ufersa.edu.br}
	\and
	Regivan H. N. Santiago \and Benjam\'in Bedregal\at
	Department of Informatics and Applied Mathematics - DIMAp, Federal University of Rio Grande do Norte - UFRN, Avenue Senador Salgado Filho, 3000, Universitary Campus of Lagoa Nova, 59.078-970, Natal, RN, Brazil
}

\date{Received: date / Accepted: date}

\maketitle

\begin{abstract}
	In this paper we introduce a class of operators on complete lattices called \textbf{Dynamic Ordered Weighted Averaging (DYOWA) functions}. These functions provide a generalized form  of an important class of aggregation functions: The Ordered Weighted Averaging (OWA) functions, whose applications can be found in several areas like: Image Processing and Decision Making. The wide range of applications of OWAs motivated many researchers to study their variations. One of them was proposed by Lizassoaim and Moreno in 2013, which extends  those functions to complete lattices. Here, we propose a new generalization of OWAs that also generalizes the operators proposed by Lizassoaim and Moreno.
\end{abstract}

\section{Introduction}

After the contributions of \citep{ZADEH1965} in the field of Fuzzy Sets, many extensions of classical mathematical theories have been developed, with several possibilities of application. Applications in areas like: Image Processing and Decision Making require some special functions capable of encoding a set of multiple values in a single value; these functions are called: \textit{Aggregation functions} \citep{Gleb2016,Bustince2013,Chen2,DUBOIS2004,Liang2014,Paternain2015,Paternain2012,Zhou2008}.

Aggregation functions can be classified into four classes: Averaging, conjunctive, disjunctive and mixed. The disjunctive and conjunctive aggregations functions provide, respectively, models for disjunctions and conjunctions in Fuzzy Logic \citep{Gleb2016,overlap,Bustince2012,DIMURO2014,DUBOIS1985,Klement,FariasJIFS2016}. On the other hand, averaging aggregation functions can be applied, for example, in fields like image processing and decision making \citep{Paternain2015,Bustince,Yager1988,Zadrozny}.

A special type of averaging aggregation  is called: {\bf Ordered Weighted Averaging} function, or simply \textbf{OWA}, function. It was developed by Yager \citep{Yager1988} with the intention to study the problem of multiple decision making, however many other applications for such operators have arisen since then \citep{Paternain2015,Llamazares,Torra,Lin}.

Some variations of \textsf{OWA}s can be found in  literature; e.g. see \citep{Chen2,Cheng,Merigo1,Merigo2,Yager2006}. All of them are defined on the set $[0,1]$. In 2013, Lizasoain and Moreno \citep{Lizasoain} generalized those operators to any complete lattice $L$.

All of such different approaches of \textsf{OWA}s have an essential common factor: They use a \textit{fixed vector of weights} $(w_1,w_2,\cdots,w_n)$ for  the final calculation. In this paper, we propose a new way of generalization of \textsf{OWA}s on complete lattices; the vector of weights is determined from the input arguments providing a ``dynamic flavour''. More precisely, \textit{the weights are variables} defined from the input vector.

We start this paper by exposing some basic concepts such as: Aggregation functions, \textsf{OWA}s, T-norms and T-conorms.  Sections 3 and 4 provide the extension of some  concepts previously listed to complete lattices; they also expose the generalization proposed by Lizasoain and Moreno. In section 5, we introduce our proposal of generalizing \textsf{OWA} for complete lattices, we study some of its properties and present some examples in different environments. We will show in this part of the paper that the \textsf{OWA} functions proposed here, as well as those proposed by Yager, are averaging functions. We will also prove that the Yager operators can be obtained as a particular case of our \textsf{OWA} operators. To conclude, we bring the section of conclusions and future works.

\vspace{-0.5cm}
\section{Aggregation Functions}

The {\bf aggregation functions} are mathematical tools that allow you to perform grouping complex information into a more simple information. More precisely, these functions are rules that associate each $n$ -dimensional input to a unique value, the output. The formal definition is presented below:

\begin{defi}[{\bf Aggregation Function}]
	A  function $A:[0,1]^n$\linebreak$\rightarrow[0,1]$ which satisfies the following properties:
	\begin{enumerate}
		\item[{\bf (A1)}] $A(0,0,\cdots, 0)=0$ and $A(1,1,\cdots,1)=1$;
		\item[{\bf (A2)}] $A(x_1,x_2,\cdots,x_n)\le A(y_1,y_2,\cdots, y_n)$ whenever $x_i\le y_i$ for all $1\le i\le n$.
	\end{enumerate}
	is called of {\bf n-ary  aggregation function}.
\end{defi}

Applications of aggregation functions can be found, for example, in decision-making problems and in the formulation of some fuzzy logic connectives (see \citep{Gleb2016}). In the following, we introduce some notations that will be used in this paper.

\begin{obs}\hfill
	\begin{enumerate}
		\item We use $\overrightarrow{x}\in X^n$ to denote the $n$-dimentional vector $\overrightarrow{x}=(x_1,x_2,\cdots,x_n)$ whose coordinates $x_i$ belong to the set $X$.
		\item Functions as $\min(\overrightarrow{x})=\min\{x_1,x_2,\cdots,x_n\}$ and \linebreak $\max(\overrightarrow{x})=\max\{x_1,x_2,\cdots,x_n\}$ are classical example of n-ary aggregations function.
		\item In order to simplify the terminology in some points of the text we use the term aggregate function instead of n-any aggregation function.
	\end{enumerate}
\end{obs}

Aggregation functions can be classified into four different types:

\begin{defi}
	Let $A:[0,1]^n\rightarrow [0,1]$ be an aggregation function. We say that $A$ is a:
	\begin{enumerate}
		\item[(i)] {\bf Averaging aggregation function} if $\min(\overrightarrow{x})\le A(\overrightarrow{x})\le \max(\overrightarrow{x})$ for any $\overrightarrow{x}\in[0,1]^n$;
		
		\item[(ii)] {\bf Conjuntive aggregation function} if $A(\overrightarrow{x})\le \min(\overrightarrow{x})$ for all $\overrightarrow{x}\in[0,1]^n$;
		
		\item[(iii)] {\bf Disjuntive aggregation function} if $A(\overrightarrow{x})\geq \max(\overrightarrow{x})$ for all $\overrightarrow{x}\in[0,1]^n$;
		
		\item[(iv)] {\bf Mixed aggregation function} if it does not belong to any of the previous classes.
	\end{enumerate}
\end{defi}

Table \ref{TableExAgg} presents some examples of  aggregation functions.

\begin{table*}[!htb]
	\centering
	\begin{tabular}{lllll}
		\hline\
		{\bf Function} & {\bf Averaging} & {\bf Conjunctive} & {\bf Disjunctive} & {\bf Mixed}\\
		\hline\
		$\!\!\min(\overrightarrow{x})$ & \ \ \ \ \ \ X & \ \ \ \ \ \ \ \ \ X & &\\
		$\max(\overrightarrow{x})$ & \ \ \ \ \ \ X & & \ \ \ \ \ \ \ \ \ X &\\
		$\displaystyle arith(\overrightarrow{x})=\frac{1}{n}\sum\limits_{i=1}^nx_i$ & \ \ \ \ \ \ X & & &\\
		$\displaystyle T_P(\overrightarrow{x})=\prod\limits_{i=1}^nx_i$ & & \ \ \ \ \ \ \ \ \ X & &\\
		$\displaystyle S_P(\overrightarrow{x})=1-\prod\limits_{i=1}^n(1-x_i)$ & & & \ \ \ \ \ \ \ \ \ X &\\
		$\displaystyle f(\overrightarrow{x})=\frac{\prod\limits_{i=1}^n x_i}{\prod\limits_{i=1}^n x_i+\prod\limits_{i=1}^n(1-x_i)}$ & & & & \ \ \ \ \ X \\
		\hline
	\end{tabular}	
	\caption{Examples of aggregation functions}\label{TableExAgg}
\end{table*}

\begin{defi}
	A function $f:[0,1]^n\rightarrow[0,1]$ satisfies the properties of:
	\begin{enumerate}
		
		\item[{\bf (IP)}] {\bf Idempotency} if $f(x,x,\cdots,x)=x$ for all $x\in[0,1]$;
		
		\item[{\bf (SP)}] {\bf Symmetry} if for any permutation $\sigma$ of the set $\{1,2,\cdots,n\}$ we have  
		$f(x_1,x_2,\cdots,x_n)= f(x_{\sigma(1)},x_{\sigma(2)},\cdots,x_{\sigma(n)})$;
		
		\item[{\bf (NP)}] {\bf Neutral element} if there is a element $e\in[0,1]$ such that for all $t\in[0,1]$ allocated in any coordinate $i$ we have to $f(e,\cdots,e,t,e\cdots, e)=t$;
		
		\item[{\bf (AP)}] {\bf Absorption} if $f$ has an absorption element $a\in[0,1]$, i.e., if for all $i\in\{1,2,\cdots,n\}$ we have to $f(x_1,\cdots,x_{i-1},$\linebreak $a,x_{i+1},\cdots, x_n)=a$;
		
		\item[{\bf (HP)}] {\bf Homogeneity} if for any $\lambda,x_1,\cdots,x_n\in[0,1]$ we have to $f(\lambda x_1,\cdots,\lambda x_n)=\lambda f(x_1,\cdots, x_n)$;
		
		\item[{\bf (ZD)}] {\bf Zero divizor} if there is  $\overrightarrow{x}=(x_1,\cdots,x_n)\in (0,1]^n$ such that $f(\overrightarrow{x})=0$;
		
		\item[{\bf (OD)}] {\bf One divizor} if there is  $\overrightarrow{x}=(x_1,\cdots,x_n)\in [0,1)^n$ such that $f(\overrightarrow{x})=1$;
		
		\item[{\bf (ASP)}] {\bf Associativity} if $n=2$ and $f(x,f(y,z))=f(f(x,y),z)$ for any $x,y,z\in[0,1]$.
	\end{enumerate}
\end{defi}

Table \ref{TablePropAgreg} presents some examples of aggregations functions which satisfy such properties.

\begin{table*}[!htb]
	\centering
	\begin{tabular}{ll}
		\hline\
		{\bf Aggregation function} & {\bf Properties} \\
		\hline\
		$\!\!\min(\overrightarrow{x})$ & {\bf (IP)}, {\bf (SP)}, {\bf (NP)}, {\bf (AP)}, {\bf (HP)} and {\bf (ASP)}\\
		$\max(\overrightarrow{x})$ & {\bf (IP)}, {\bf (SP)}, {\bf (NP)}, {\bf (AP)}, {\bf (HP)} and {\bf (ASP)}\\
		$\displaystyle arith(\overrightarrow{x})=\frac{1}{n}\sum\limits_{i=1}^nx_i$ & {\bf (IP)}, {\bf (SP)}, {\bf (HP)} and {\bf (ASP)}\\
		$\displaystyle T_P(\overrightarrow{x})=\prod\limits_{i=1}^nx_i$ & {\bf (SP)}, {\bf (NP)}, {\bf (AP)} and {\bf (ASP)}\\
		$\displaystyle S_P(\overrightarrow{x})=1-\prod\limits_{i=1}^n(1-x_i)$ & {\bf (SP)}, {\bf (NP)}, {\bf (AP)} and {\bf (ASP)}\\
		$\displaystyle f(\overrightarrow{x})=\frac{\prod\limits_{i=1}^n x_i}{\prod\limits_{i=1}^n x_i+\prod\limits_{i=1}^n(1-x_i)}$ & {\bf (SP)} and {\bf (AP)}\\
		\hline
	\end{tabular}	
	\caption{Properties of aggregation functions}\label{TablePropAgreg}
\end{table*}

\subsection{OWA Functions}

The Ordered Weighted Averaging -- OWA function, defined by Yager in \citep{Yager1988}, constitute an important family of averaging aggregation functions, 
which have been widely studied by many researchers around the world, motivated by its wide range of applications. Applications of OWA can be found, for 
example, in image processing \citep{Paternain2015,Bustince,Zadrozny}, in neural networks \citep{Amin1,Amin2,Emrouznejad} and in decision 
making \citep{Cheng,Ahn1,DeMiguel2016}. The definition of this important class of functions is presented below:

\begin{defi}[{\bf OWA Funtion}]
	Given a $n$-dimentional vector of weights \footnote{A $n$-dimentional vector of weights is $\overrightarrow{w}=(w_1,\cdots,w_n)\in [0,1]^n$ such that $\sum\limits_{i=1}^nw_i=1$.} $\overrightarrow{w}=(w_1,\cdots,w_n) \in [0,1]^n$, the function
	$$OWA_{\overrightarrow{w}}(\overrightarrow{x})=\sum\limits_{i=1}^nw_ix_{(i)},$$
	where $(x_{(1)},x_{(2)},\cdots,x_{(n)})=Sort(\overrightarrow{x})$ is the descending ordernation of the vector $(x_1,x_2,\cdots,x_n)$, is called of {\bf Ordered Weighted Averaging} function or simply {\bf OWA function}.
\end{defi} 

It is not difficult to show that for any vector of weights $\overrightarrow{w}=(w_1,\cdots,w_n)$, the function $OWA_{\overrightarrow{w}}$ is an averaging aggregation function. Furthermore, $OWA's$ are continuous functions which satisfy: {\bf (IP)}, {\bf (SP)} and {\bf (HP)}, but do not:  {\bf (ZD)} and {\bf (OD)}. They are parametric  functions; namely: Depending on the vector of weights it will  simulate an average aggregation function. Below, we present some examples:

\begin{exem}\hfill\
	\begin{enumerate}
		\item $\min(x_1,\cdots,x_n)$ is obteined by vector of weights $\overrightarrow{w}=(0,\cdots,0,1)$;
		
		\item $\max(x_1,\cdots,x_n)$ can be obteined by vector of weights $\overrightarrow{w}=(1,0\cdots,0)$;
		
		\item $arith(x_1,\cdots,x_n)$ is the OWA function with \linebreak$\overrightarrow{w}=\left(\frac{1}{n},\cdots,\frac{1}{n}\right)$;
		
		\item The median, $$med(x_1,\cdots,x_n)=\left\{\begin{array}{ll} \frac{x_{(k)}+x_{(k+1)}}{2}, & \mbox{ if } n=2k\\ x_{\left(\frac{k+1}{2}\right)}, & \mbox{ otherwise} \end{array}\right.$$ can be found from:
		\begin{itemize}
			\item If $n$ is odd, then $w_i=0$ for all $i\ne\lceil\frac{n}{2}\rceil$ and $w_{\lceil n/2\rceil}=1$.
			\item If $n$ is even, then $w_i=0$ for all $i\ne \lfloor\frac{n+1}{2}\rfloor$ and $i\ne \lceil\frac{n+1}{2}\rceil$, and $w_{\lceil (n+1)/2\rceil}=w_{\lfloor (n+1)/2 \rfloor}=\frac{1}{2}$.
		\end{itemize}
	\end{enumerate}
\end{exem}

\subsection{t-norms and t-conorms}

Some aggregation functions provide models for conjunctions and disjunctions in fuzzy logic. These operators are called respectively of t-norms and t-conorms:

\begin{defi}[{\bf t-norms}]
	A {\bf t-norm} is a function \linebreak$T:[0,1]\times [0,1]\rightarrow [0,1]$ which satisfies:
	\begin{enumerate}
		\item[{\bf (T1)}] $T(x,1)=x$ for all $x\in[0,1]$;
		\item[{\bf (T2)}] $T(x,y)=T(y,x)$ for any $x,y\in[0,1]$;
		\item[{\bf (T3)}] $T(x,T(y,z))=T(T(x,y),z)$ for any $x,y,z\in[0,1]$;
		\item[{\bf (T4)}] $T(x,y)\le T(x,z)$ whenever $x\le z$.
	\end{enumerate}
\end{defi}

\begin{defi}[{\bf t-conorms}]
	A {\bf t-conorm} is a function \linebreak$S:[0,1]\times [0,1]\rightarrow [0,1]$ such that:
	\begin{enumerate}
		\item[{\bf (S1)}]$S(x,0)=x$ for any $x\in[0,1]$;
		\item[{\bf (S2)}] $S(x,y)=S(y,x)$ for all $x,y\in[0,1]$;
		\item[{\bf (S3)}] $S(x,S(y,z))=S(S(x,y),z)$ for all $x,y,z\in[0,1]$;
		\item[{\bf (S4)}] $S(x,y)\le S(x,z)$ whenever $x\le z$.
	\end{enumerate}
\end{defi}

T-norms are conjunctive and t-conorms are disjunctive aggregation functions. Table \ref{TableExTNorm} contains some examples of t-norms and t-conorms. 
The reader can find in \citep{Klement} a deeper insight about t-norms and t-conorms on $[0,1]$ and  in \citep{Bedregal2006,Mesiar,DeCooman1994,PALMEIRA2014} on some class of the lattices.

\begin{table*}[!htb]
	\centering
	\begin{tabular}{ll}
		\hline\
		{\bf t-norms} & {\bf t-conorms}\\
		\hline\
		$T_{min}(x,y)=\min(x,y)$ & $S_{max}(x,y)=\max(x,y)$\\
		$T_P(x,y)=x.y$ & $S_P(x,y)=x+y-xy$\\
		$T_{LK}(x,y)=\max(x+y-1,0)$ & $S_{LK}(x,y)=\min(x+y,1)$\\
		$T_D(x,y)=\left\{\begin{array}{ll} 0, &\mbox{ if } x,y\in[0,1)\\ \min(x,y), & \mbox{ otherwise } \end{array}\right.$ & $S_D(x,y)=\left\{\begin{array}{ll} 1, & \mbox{ if } x,y\in(0,1]\\ \max(x,y), & \mbox{ otherwise }\end{array}\right.$\\
		\hline
	\end{tabular}
	\caption{Examples of t-norms and t-conorms}\label{TableExTNorm}
\end{table*}

\vspace{-0.5cm}
\section{Aggregations, t-norms and t-conorms for complete lattices}

A {\bf complete lattice} is a  partial order, $\langle L, \le_L \rangle$, in which any subset $S\subseteq L$ has supremum and infimum elements, denoted 
respectively by $\bigvee S$ and $\bigwedge S$ \citep{Birkhoff,Gierz}. Complete lattices are bounded; i.e. they have  top, $\top_L$,  and bottom elements, $\bot_L$.

The properties {\bf (IP)}, {\bf (SP)}, {\bf (NP)}, {\bf (AP)}, {\bf (HP)}, {\bf (ZD)}, {\bf (OD)} and {\bf (ASP)} can be extended to lattices, as well as:  
aggregations, t-norms and t-conorms. 

\begin{defi}
	An isotonic function\footnote{A function $f:L^n\rightarrow L$ is isotonic if, $f(x_1,\cdots,x_n)\le_L f(y_1,\cdots,y_n)$, whenever $x_i\le_L y_i$, for all $(x_1,\cdots,x_n), (y_1,\cdots,y_n)\in L^n$.} $A:L^n\rightarrow L$ such that:
	\begin{enumerate}
		\item[{\bf (A1)}] $A(\bot_L,\cdots,\bot_L)=\bot_L$;
		\item[{\bf (A2)}] $A(\top_L,\cdots,\top_L)=\top_L$
	\end{enumerate} 
	is called of aggregation function on $L$.
\end{defi}

\begin{defi}[\citep{Mesiar}]
	\begin{enumerate}
		\item[]
		\item[(1)] An isotonic binary operator $\otimes:L\times L\rightarrow L$ which satisfies {\bf (SP)} and {\bf (ASP)}, and have $\top_L$ as neutral element is a t-norm on $L$.
		\item[(2)]  An isotonic binary operator $\oplus:L\times L\rightarrow L$ that satisfies {\bf (SP)} and {\bf (ASP)}, and have $\bot_L$ as neutral element is a t-conorm on $L$.
	\end{enumerate}
\end{defi}

The associativity of $\oplus$ and $\otimes$ allows us to define n-any operators, as follow:
$$\bigoplus\limits_{i=1}^n x_i = ((((x_1\oplus x_2)\oplus x_3)\oplus\cdots \oplus )x_{n-1})\oplus x_n;$$
\begin{center}
	and
\end{center}
$$\bigotimes\limits_{i=1}^n x_i = ((((x_1\otimes x_2)\otimes x_3)\otimes\cdots) \otimes x_{n-1})\otimes x_n.$$

\begin{prop}\label{prop:TnormProp}
	Let $\oplus,\otimes:L\times L\rightarrow L$ be a t-norm and a t-conorm on a complete lattice $L$. Then, for any $a,b,c\in L$ are valid:
	\begin{enumerate}
		\item[(i)] $a\otimes b\le_L a\otimes\top_L=a$ and $a=a\oplus\bot_L\le_L a\oplus b$;
		\item[(ii)] $a\otimes b\le_L a \wedge b$ and $a\vee b\le_L a\oplus b$;
		\item[(iii)] $a\otimes(b\otimes c))=(a\otimes b)\otimes c$ and $a\oplus (b\oplus c)=(a\oplus b)\oplus c$; 
		\item[(iv)] $a\otimes \bot_L=\bot_L=\bot_L\otimes a$.
	\end{enumerate}
\end{prop}

\vspace{-0.5cm}
\section{OWA operators for complete lattices}

Several variations of OWA functions  defined on the interval  $[0,1]$ can be found in literature; e.g.  \textbf{IGOWA}, \textbf{IGCOWA}, 
\textbf{POWA} and \textbf{cOWA}  \citep{Merigo1,Merigo2,Yager2006,Chen}. Another approach is due to Lizasoain and Moreno \citep{Lizasoain} which 
generalized OWAs to complete lattices.

\begin{defi}[Definition 3.3 of \citep{Lizasoain}]
	Let $L$  be a complete lattice and $\otimes,\oplus:L\times L\rightarrow L$ be a t-norm and a t-conorm. We say that $(w_1,\cdots,w_n)\in L^n$ is a:
	\begin{enumerate}
		\item[(i)] vector of weights on $\langle L,\bigoplus,\otimes\rangle$ whenever $\bigoplus\limits_{i=1}^n w_i=\top_L$;
		\item[(ii)] distributive vector of weights on $\langle L,\bigoplus,\otimes \rangle$ whenever it satisfies the (i) and  $$a\otimes \left(\bigoplus\limits_{i=1}^n w_i\right)=\bigoplus\limits_{i=1}^n(a\otimes w_i) \mbox{, for any } a\in L^n.$$
	\end{enumerate}
\end{defi}

\begin{obs}
	If $L$ is a complete lattice and $x\oplus y=x\vee y$ and $x\otimes y=x\wedge y$, then $(w_1,\cdots,w_n)$ is a:
	\begin{enumerate}
		\item[(1)] vector of weights if, and only if, $w_1\vee \cdots\vee w_n=\top_L$;
		\item[(2)] distributive vector of weights if, and only if, satisfies $(2)$ and $a\wedge(w_1\vee\cdots\vee w_n)=(a\wedge w_1)\vee\cdots\vee(a\wedge w_n)$, for all $a\in L$.
	\end{enumerate}
\end{obs}

To calculate the output of an OWA (in the sense of Yager), we need to sort the {\it n}-dimensional input vector in a decreasing way. This process is 
always possible when the underlying complete lattice is  the linear order, but there are complete lattices with pairs of non-comparable elements. 
For this reason, we need to define an auxiliary vector from the input vector. This is done using the following Lemma:

\begin{lema}[Lemma 3.1 of \citep{Lizasoain}]
	Let $L$ be a complete lattice. For any $(a_1,\cdots,a_n)\in L^n$, consider the following values:
	\begin{itemize}
		\item $b_1=a_1\vee\cdots \vee a_n$
		\item $b_2=[(a_1\wedge a_2)\vee\cdots\vee(a_1\wedge a_n)]\vee[(a_2\wedge a_3)\vee\cdots(a_2\wedge a_n)]\vee\cdots\vee[a_{n-1}\wedge a_n]$
		\item $b_3 = \bigvee \{a_1\wedge a_2\wedge a_3, a_1\wedge a_2\wedge a_4, a_1\wedge a_3\wedge a_4,\cdots, a_{n-2}\wedge a_{n-1}\wedge a_n\}$
		\item $\vdots$
		\item $b_k=\bigvee\limits_{(j_1,\ldots,j_k)\subseteq \{1,\cdots,n\}} a_{j_1}\wedge\cdots \wedge a_{j_k}$
		\item $\vdots$
		\item $b_n=a_1\wedge\cdots \wedge a_n$.
	\end{itemize}
	Then, $a_1\wedge\cdots \wedge a_n=b_n\le b_{n-1}\le\cdots\le b_1=a_1\vee\cdots\vee a_m$.  If  $\{a_1,\cdots,a_n\}$ is totally ordered, then there is a permutation, $\sigma$, for the set $\{1,\cdots,n\}$, such that $(b_1,\cdots,b_n)=(a_{\sigma(1)},\cdots,a_{\sigma(n)})$.
\end{lema}

A proof of this Lemma can be found in \citep{Lizasoain}. To simplify the notations we use the following definition:

\begin{defi}
	If $L$ is a complete lattice, then the function $\mathscr{L}\!\mathscr{M}:L^n\rightarrow L^n$ defined by:
	$$\mathscr{L}\!\mathscr{M}(a_1,a_2,\cdots,a_n)=(b_1,b_2,\cdots,b_n),$$
	where $(b_1,b_2,\cdots,b_n)$ is the $n$-dimentional vector obtained according to Lemma 1, is called of Lizassoain-Moreno function. 
\end{defi}

\begin{exem}
	If $L$ is a complete lattice and $a_1,a_2,a_3,a_4\in L$, then:
	\begin{eqnarray*}
		b_1 & = & a_1\vee a_2\vee a_3\vee a_4\\
		b_2 & = & \bigvee\{a_1\wedge a_2, a_1\wedge a_3, a_1\wedge a_4, a_2\wedge a_3,\\ & & \ \ \ \ \ \ \ \ \ \ \ \ \ \ \ \ \ \ \ \ \ \ \ \ \ \ \ \ \ \ \ \ \ \ \ \ \ \ \ \ \ \ \ a_2\wedge a_4, a_3\wedge a_4\}\\
		b_3 & = & \bigvee \{a_1\wedge a_2\wedge a_3, a_1\wedge a_2\wedge a_4, a_1\wedge a_3\wedge a_4,\\ & & \ \ \ \ \ \ \ \ \ \ \ \ \ \ \ \ \ \ \ \ \ \ \ \ \ \ \ \ \ \ \ \ \ \ \ \ \ \ \ \ \ \ \ \ \ \ \ \ \  a_2\wedge a_3\wedge a_4 \}\\
		b_4 & = & a_1\wedge a_2\wedge a_3\wedge a_4
	\end{eqnarray*}
\end{exem}

In the following, we list some properties of Lizassoain-Moreno function:

\begin{prop}[{\bf Properties of Lizassoain-Moreno function}]
	If $L$ is a complete lattice, then:
	\begin{enumerate}
		\item[(i)] If $(a_1,\cdots,a_n)\in L^n$ is such that any pair of coordinates $a_i$ and $a_j$, with $i\ne j$, is comparable, then 
		$\mathscr{L}\!\mathscr{M}(a_1,\cdots,a_n)=(a_{\sigma(1)},\cdots,a_{\sigma(n)})$, where $\sigma$ is a permutation on the set 
		$\{1,2,\cdots,n\}$ such that $(a_{\sigma(1)},\ldots,$\linebreak$a_{\sigma(n)})=Sort(a_1,\ldots,a_n)$.
		\item[(ii)] If $L$ is a linear order, then for all $(a_1,\cdots,a_n)\in L^n$ there is a permutation $\sigma$ on the set 
		$\{1,2,\cdots,n\}$ such that $\mathscr{L}\!\mathscr{M}(a_1,\cdots,a_n)=(a_{\sigma(1)},\cdots,a_{\sigma(n)})$.
		\item[(iii)] If $a_1\geq a_2\geq\cdots\geq a_n$, then $\mathscr{L}\!\mathscr{M}(a_1,\cdots,a_n)=(a_1,\cdots,a_n)$. 
		\item[(iv)] $\mathscr{L}\!\mathscr{M}\circ\mathscr{L}\!\mathscr{M}=\mathscr{LM}$
		\item[(v)] For any permutation $\sigma$ for  $\{1,\cdots,n\}$ and for all $(a_1,\cdots,$\linebreak$a_n)\in L^n$,  $\mathscr{L}\!\mathscr{M}(a_{\sigma(1)},\cdots,a_{\sigma(n)})=\mathscr{L}\!\mathscr{M}(a_1,\cdots,a_n)$.
	\end{enumerate}
\end{prop}

\begin{proof}\hfill
	\begin{enumerate}
		\item[(i)] Straightforward from Lemma 1.
		\item[(ii)] Since $L$ is a linear order, then any $(a_1,\cdots,a_n)$ satisfies $(i)$. Thus, for any $(a_1,\cdots,a_n)$ there is a permutaion $\sigma$ such that $\mathscr{L}\!\mathscr{M}(a_1,\cdots,a_n)=(a_{\sigma(1)},\cdots,a_{\sigma(n)})$.
		\item[(iii)] Straightforward from definition of Lizassoain and Moreno funcion.
		\item[(iv)] As $\mathscr{L}\!\mathscr{M}(a_1,\cdots,a_n)=(b_1,\cdots,b_n)$, where $b_1\geq b_2\geq \cdots\geq b_n$ then, by $(iii)$ follows that for all $(a_1,\cdots,a_n)\in L^n$,
		\begin{eqnarray*}
			\mathscr{L}\!\mathscr{M}(\mathscr{L}\!\mathscr{M}(a_1,\cdots,a_n)) & = &\mathscr{L}\!\mathscr{M}(b_1,\cdots,b_n)\\
			& = & b_1,\cdots,b_n\\
			& = & \mathscr{L}\!\mathscr{M}(a_1,\cdots,a_n)
		\end{eqnarray*}
		
		\item[(v)] Let be $\mathscr{L}\!\mathscr{M}(a_1,\cdots,a_n)=(b_1,\cdots,b_n)$ and $\mathscr{L}\!\mathscr{M}(a_{\sigma(1)},$ $\cdots,a_{\sigma(n)})=(c_1,\cdots,c_n)$. By definition, we have to:
		\small\begin{eqnarray*}
			c_k & = & \bigvee_{(j_1,\ldots,j_k)\subseteq \{1,\cdots,n\}} a_{\sigma(j_1)}\wedge\cdots\wedge a_{\sigma(j_k)}\\
			& = & \bigvee_{(j_1,\ldots,j_k)\subseteq \{1,\cdots,n\}} a_{j_1}\wedge\cdots\wedge a_{j_k}\\
			& = & b_k
		\end{eqnarray*}
		Therefore, $\mathscr{L}\!\mathscr{M}(a_{\sigma(1)},\cdots,a_{\sigma(n)})=\mathscr{L}\!\mathscr{M}(a_1,\cdots,a_n)$.
	\end{enumerate}
\end{proof}

Now that we have defined the Lizassoain-Moreno function and know some of its properties, it is possible to define a generalized version of \textsf{OWA}s for complete lattices:

\begin{defi}[Definition 3.5 of \citep{Lizasoain}]
	Let $\overrightarrow{w}{=}(w_1\ldots,w_n)$ $\in L^n$ be a distributive vector of weights in a complete lattice $\langle L,\bigoplus,\otimes\rangle$. For any $\overrightarrow{a}=(a_1,\cdots,a_n)\in L^n$, consider the totally decreasing ordered vector $(b_1,\cdots,b_n)=\mathscr{L}\!\mathscr{M}(a_1,\cdots,a_n)$. The Lizasoain-Moreno \textsf{OWA} function  associated with $\overrightarrow{w}$ and the triplet $\langle L, \bigoplus,\otimes\rangle$ is 
	\begin{equation}
		LMOWA_{\overrightarrow{w}}(\overrightarrow{a})=\bigoplus\limits_{i=1}^n(w_i\otimes b_i)
	\end{equation}
\end{defi}

Examples and properties of Lizasoain-Moreno OWA can be found in \citep{Lizasoain}, it is noteworthy  that Yager's OWA is a particular case of Lizasoain-Moreno's OWA:
	
\begin{teo}\label{th:YagerIsLM}
	Every  Yager's \textsf{OWA} is an \textsf{LMOWA}. 
\end{teo} 

\begin{proof}
	Let $L=[0,1]$, $\otimes=T_p$, $\oplus=S_{LK}$ and  $\overrightarrow{w}=(w_1,\cdots,w_n)\in[0,1]^n$ be a vector of weights. Since, \linebreak $S_{LK}(w_1,\cdots,w_n)=\min(w_1+\cdots+w_n,1)=1$ and 
	$c=T_P(c,S_{LK}(w_1,\cdots,w_n))=S_{LK}(T_P(c,w_i))$, for all $c\in [0,1]$,
	then $\overrightarrow{w}=(w_1,\cdots,w_n)\in[0,1]^n$ is a distributive vector of weights.
	
	To prove that $LMOWA_{\overrightarrow{w}}$ coincides with  the Yager's \textsf{OWA}, first note that the Proposition 2 ensures that for any input vector $(x_1,\cdots,x_n)$ there is a permutation $\sigma$ on $\{1,2,\cdots,n\}$ such that $(b_1,\cdots,b_n)=(x_{\sigma(1)},\cdots,x_{\sigma(n)})$. Besides,
	$$x_{\sigma(1)}\geq x_{\sigma(s)}\geq \cdots \geq x_{\sigma(n)},$$
	that is, $(b_1,\cdots,b_n)=(x_{(1)},\cdots, x_{(n)})$. Therefore, for all $(x_1,\cdots,x_n)\in[0,1]^n$ it is verified that:
	\begin{eqnarray*}
		LMOWA_{\overrightarrow{w}}(x_1,\cdots,x_n)& = &\bigoplus\limits_{i=1}^n(w_i\otimes b_i)\\
		& = & S_{LK}(w_1\otimes b_1,\ldots, w_n\otimes b_n)\\
		& = & \min(\sum\limits_{i=1}^n(w_i\cdot x_{\sigma(i)}),1)\\
		& = & \sum\limits_{i=1}^n(w_i\cdot x_{\sigma(i)})\\
		& = & OWA_{\overrightarrow{w}}(x_1,\cdots,x_n)\\
		&    &\\
	\end{eqnarray*}

\end{proof}

\begin{obs}
	\begin{enumerate}
		\item[]
		\item \textsf{OWA}'s satisfies the properties {\bf (IP)} and {\bf (SP)}. Futhermore, for any distributive vector of weights $\overrightarrow{w}\in L^n$ and all $(a_1,\cdots,a_n)\in L^n$ we have
		$$a_1\wedge \cdots \wedge a_n\le OWA_{\overrightarrow{w}}(a_1,\cdots,a_n)\le a_1\vee\cdots\vee a_n$$
	\end{enumerate}
\end{obs}

Now, observe that both: Yager's and Lizasoain-Moreno's \textsf{OWA}s   are obtained from a unique \textit{fixed} vector of weights $\overrightarrow{w}$. In \citep{FariasDYOWA,Farias2} we propose a generalization of Yager's \textsf{OWA}, in such a way that the weights are not fixed. In this sense, we propose here a generalization of Lizasoain-Moreno's \textsf{OWA} taking into account  nonfixed weights.

\vspace{-0.5cm}
\section{Dynamic Ordered Weighted Averaging Functions}

In the sequel we propose and investigate a generalized form of \textsf{OWA} for complete lattices; they are named  {\bf Dynamic Ordered Weighted Averaging (DYOWA)} functions. The \textbf{DYOWA} functions generalize both Yager's and  Lizasoain-Moreno's \textsf{OWA}. In order to to introduce them,  we need first to define the notion of  \textit{weights function}.

\begin{defi}[{\bf Weight function}]
	Let be the structure $\langle L,\bigoplus,$\linebreak $\otimes\rangle$, where $L$ is a complete lattice, $\otimes:L^2\rightarrow L$ is a t-norm and $\oplus:L^2\rightarrow L$ is a t-conorm. A  finite family of functions $\Gamma =\{f_i:L^n\rightarrow L:\ i=1,2,\cdots,n\}$ is called of {\bf weight function family} whenever for all $\overrightarrow{a}\in L^n$, 
	$(f_1(\overrightarrow{a}),\cdots,f_n(\overrightarrow{a}))$ is a vector of weights 
	on $\langle L,\bigoplus , \otimes\rangle$. The natural function $f:L^n\rightarrow L^n$, s.t. $f(\overrightarrow{a})=(f_1(\overrightarrow{a}),\cdots,$\linebreak $f_n(\overrightarrow{a}))$ is called {\bf weight function} on $\Gamma$\ \footnote{Or just weight function whenever $\Gamma$ is clear in the context.}.
	If furthermore,  for all  $c\in L$,
	$$c\otimes\left( \bigoplus\limits_{i=1}^n f_i(\overrightarrow{a}) \right)=\bigoplus\limits_{i=1}^n(c\otimes f_i({\overrightarrow{a}}))$$
	$f$ is called {\bf distributive weight function} on $\Gamma$ or simply that $\Gamma$  is a \textbf{distributive family}. 
\end{defi}

	\begin{obs}
		Since $\bigoplus\limits_{i=1}^n f_i(\overrightarrow{a})=\top_L$ then, by (T1), $c=c\otimes\left( \bigoplus\limits_{i=1}^n f_i(\overrightarrow{a}) \right)$.
	\end{obs}

The following are some examples of distributive families:

\begin{exem}
	Let $L$ be a complete lattice, $\oplus =\wedge$ and $\otimes=\vee$. Then a weight family of functions $\Gamma =\{f_i:L^n\rightarrow L:\ i=1,2,\cdots,n\}$ must satisfy
	$$\bigvee\limits_{i=1}^nf_i(\overrightarrow{a})=\top_L.$$
	In particular, if for all $\overrightarrow{a}\in L^n$, $f_1(\overrightarrow{a})=\top_L$ and $f_k(\overrightarrow{a})=\bot_L$ for $1<k\le n$, then $\Gamma_1=\{f_1,\cdots,f_n \}$ is a weight function. Furthermore,
	$$c\wedge\left( \bigvee\limits_{i=1}^n f_i(\overrightarrow{a})\right)=c\wedge \top_L=c=c\vee \bot_L=\bigvee\limits_{i=1}^n(c\wedge f_i(\overrightarrow{a})),$$
	that is, $\Gamma$ is a distributive family. Analogously, if $g_n(\overrightarrow{a})=\top_L$ and $g_k(\overrightarrow{a})=\bot_L$ for $\ 1\le k<n$, then $\Gamma_2=\{g_1,\cdots,g_n\}$ is also a distributive family.
\end{exem}

\begin{exem}
	Consider $L=[0,1]$, $\otimes=T_P$ and $\oplus=S_{LK}$. A distributive family 
	$\Gamma =\{f_i:L^n\rightarrow L:\ i=1,2,\cdots,n\}$  must satisfy:
			$$f_1(\overrightarrow{x})+\cdots+f_n(\overrightarrow{x})=1,$$
	and
	$$c= \min(c\cdot	f_1(\overrightarrow{x})+\cdots+ c\cdot f_n(\overrightarrow{x}),1)$$
	for any $c\in [0,1]$ and each $\overrightarrow{x}\in[0,1]^n$. In particular, $\Gamma=\{f_1,\cdots,f_n\}$, with
	$$f_i(x_1,\cdots,x_n)=\left\{\begin{array}{ll}\frac{1}{n}, & \mbox{ if } x_1=\cdots=x_n\\ \frac{x_1}{\sum\limits_{i=1}^n x_i}, & \mbox{ otherwise} \end{array} \right.$$
	satisfies these properties. Therefore, $\Gamma$ is a distributive family.
\end{exem}

\begin{exem}
	Let be $L=\langle\mathbb{I}[0,1],\le_{KM}\rangle$, where $\mathbb{I}[0,1]=\{[a,b]:\ 0\le a\le b\le 1 \}$ and $\le_{KM}$ is the 
	{\it Kulisch-Miranker} partial order \citep{Kulisch}, i.e., $[a,b]\le_{KM}[c,d] \iff a\le c \mbox{ and } b\le d$. Moreover, 
	$[a,b]\oplus [c,d] = [T_P(a,c), T_P(b,d)]$ and $[a,b]\otimes[c,d] = [S_{LK}(a,c), S_{LK}(b,d)]$ 
	are a t-norm and a t-comorm, respectively, on $L$ (See Bedregal and Takahashi \citep{BedregalTakahashi2006}). It is easy to verify that the finite family of functions $\Gamma$ formed by $f_i:\mathbb{I}[0,1]^n\rightarrow\mathbb{I}[0,1]^n$ given by $f_i(I_1\cdots,I_n)=\left[\frac{1}{n},\frac{1}{n}\right]$ (constant functions) provides a distributive family.
\end{exem}

Now we can define  our proposed generalized form of \textsf{OWA}.

\begin{defi}[{\bf DYOWAs}]
	Given a complete lattice $L$, a t-norm $\otimes$, a t-conorm $\oplus$ and a weigtht function family $\Gamma=\{f_i:L^n\rightarrow L:\ i=1,2,\cdots,n\}$, we call of {\bf Dynamic Ordered Weighted Averaging Function  (DYOWA)}, the function:
	$$DYOWA_{\Gamma}(\overrightarrow{a})=\bigoplus\limits_{i=1}^n(f_i(\overrightarrow{a})\otimes b_i),$$
	where $(b_1,\cdots,b_n)=\mathscr{L}\!\mathscr{M}(a_1,\cdots,a_n)$.
\end{defi}

Below we will present some examples of \textbf{DYOWA} functions.

\begin{exem}
	Let $\oplus$, $\otimes$, $\Gamma_1$ and $\Gamma_2$ defined in Example 3, then $DYOWA_{\Gamma_1}$ and $DYOWA_{\Gamma_2}$ are:
	\begin{eqnarray*}
		DYOWA_{\Gamma_1}(\overrightarrow{a}) & = & \bigvee\limits_{i=1}^n(f_i(\overrightarrow{a})\wedge b_i)\\
		& = & (\top_L\wedge b_1)\vee(\bot_L\wedge b_2)\vee\cdots\\
		& & \hspace{3.35cm} \vee(\bot_L\wedge b_n)\\
		& = & b_1\vee \bot_L\vee\cdots\vee \bot_L\\
		& = & b_1\\
		&= & a_1\vee a_2\vee\cdots \vee a_n
	\end{eqnarray*} 
	and
	\begin{eqnarray*}
		DYOWA_{\Gamma_2}(\overrightarrow{a}) &= & \bigvee\limits_{i=1}^n(g_i(\overrightarrow{a})\wedge b_i)\\
		& = & (\bot_L\wedge b_1)\vee\cdots\vee(\bot_L\wedge b_{n-1})\\
		& & \hspace{2.8cm} \vee(\top_L\wedge b_n)\\
		& = & \bot_L\vee \cdots \vee \bot_L \vee b_n\\
		& = & b_n\\
		& = & a_1\wedge a_2\wedge\cdots \wedge a_n
	\end{eqnarray*}
	
	In addition, $DYOWA_{\Gamma_1}$ and $DYOWA_{\Gamma_2}$ are isotonic functions which satisfy {\bf (IP)}, {\bf (SP)}, {\bf (NP)}, {\bf (AP)} and {\bf (ASP)}, but do not satisfy {\bf (ZD)} and {\bf (OD)}.
\end{exem}

\begin{exem}
	If $\oplus$, $\otimes$ and $\Gamma$ are defined as in Example 4, then
	$$DYOWA_{\Gamma}(x_1,\cdots,x_n)=\sum\limits_{i=1}^n\frac{x_i^2}{\sum\limits_{j=1}^n x_j}=
	\frac{\sum\limits_{i=1}^n x_i^2}{\sum\limits_{j=1}^n x_j}.$$
	Futhermore, $DYOWA_{\Gamma}$ satisfies {\bf (IP)} and {\bf (SP)}, but do not  {\bf (NP)}, {\bf (HP)}, {\bf (ZD)}, {\bf (OD)} and {\bf (AP)}. It is important to note that, when $n=3$, this \textsf{DYOWA} function is not monotonic, since $DYOWA(0.5,0.2,0.1)=0.375$ and $DYOWA(0.5,$ \linebreak $0.22,0.2)=0.368$.
\end{exem}

\begin{exem}
	For $L=\langle \mathbb{I}[0,1],\le_{KM}\rangle$, $\oplus$, $\otimes$ and $\Gamma$ provided in Example 5 a formulae to calculate $DYOWA_\Gamma$ for any dimension $n$ is difficult to find. However, for $n=2$ this can be done by observing that for  $[\underline{x},\overline{x}],[\underline{y},\overline{y}]\in \mathbb{I}[0,1]$:
	$$[\underline{x},\overline{x}]\wedge_{KM}[\underline{y},\overline{y}] =[\min(\underline{x},\underline{y}),\min(\overline{x},\overline{y})]$$
	and
	$$[\underline{x},\overline{x}]\vee_{KM}[\underline{y},\overline{y}] =[\max(\underline{x},\underline{y}),\max(\overline{x},\overline{y})]$$
	Thus, for any $([\underline{x},\overline{x}],[\underline{y},\overline{y}])\in\mathbb{I}[0,1]^2$, we obtain the totally ordered vector $([\underline{b_1},\overline{b_1}],[\underline{b_2},\overline{b_2}])$ as follows:
	$$[\underline{b_1},\overline{b_1}]=[\max(\underline{x},\underline{y}),\max(\overline{x},\overline{y})]$$
	and
	$$[\underline{b_2},\overline{b_2}]=[\min(\underline{x},\underline{y}),\min(\overline{x},\overline{y})].$$
	Therefore, $	DYOWA_{\Gamma}([\underline{x},\overline{x}],[\underline{y},\overline{y}])$ is
	{\scriptsize\begin{eqnarray*}
			\left[\min\left(\frac{\max(\underline{x},\underline{y})+\min(\underline{x},\underline{y})}{2},1\right), \min\left(\frac{\max(\overline{x},\overline{y})+\min(\overline{x},\overline{y})}{2},1\right)\right]\\
			= \left[\min\left(\frac{\underline{x}+\underline{y}}{2},1\right), \min\left(\frac{\overline{x}+\overline{y}}{2},1\right) \right]\\
			= \left[\frac{\underline{x}+\underline{y}}{2},\frac{\overline{x}+\overline{y}}{2}\right]
		\end{eqnarray*}
	}
	
	We can easily verify that $DYOWA_\Gamma$ is an isotonic function which satisfies {\bf (IP)}, {\bf (SP)}, {\bf (HP)}, but does not  {\bf (NP)}, {\bf (ZD)}, {\bf (OD)}, {\bf (AP)} and {\bf (ASP)}.
\end{exem}

In what follows, we prove some general properties of \textsf{DYOWA} functions.

\vspace{-0.5cm}
\section{Properties of DYOWA Functions}

In the Example 7 we provided a \textsf{DYOWA} function which is not  isotonic and hence is not an aggregation. In this section we show some other underlying properties of such functions.

To ensure the monotonicity of a \textsf{DYOWA} function from the weight functions is not an easy task, however the  boundary conditions are characterized by the next theorem.

\begin{teo}
	Given  $\langle L, \bigoplus, \otimes \rangle$ and a weighted function family $\Gamma=\{f_i:L^n\rightarrow L:\ i=1,2,\cdots,n\}$, then $DYOWA_{\Gamma}$ is such that $DYOWA_{\Gamma}(\bot_L,\cdots,\bot_L)=\bot_L$ and\linebreak $DYOWA_{\Gamma}(\top_L,$ $\cdots,\top_L)=\top_L$.
\end{teo}

\begin{proof}
	By definition,
	\begin{center}
		\small
		$DYOWA_{\Gamma}(\bot_L,\cdots,\bot_L)=\bigoplus\limits_{i=1}^n(f_i(\bot_L,\cdots,\bot_L)\otimes b_i),$
	\end{center}
	but $b_i=\bot_L$ for all $i=1,2,\ldots,n$. So, according to item (iv) of Proposition \ref{prop:TnormProp}, $a\otimes\bot_L=\bot_L$:
	\begin{center}
		\small
		\begin{eqnarray*}
			DYOWA_{\Gamma}(\bot_L,\cdots,\bot_L)&=&\bigoplus\limits_{i=1}^n(f_i(\bot_L,\cdots,\bot_L)\otimes \bot_L) \\
			&=& \bigoplus\limits_{i=1}^n \bot_L =\bot_L
		\end{eqnarray*}
	\end{center}
	%
	On the other hand,
	{\small\begin{eqnarray*}
			DYOWA_{\Gamma}(\top_L,\cdots,\top_L)& =& \bigoplus\limits_{i=1}^n(f_i(\top_L,\cdots,\top_L)\otimes b_i)\\
			& = &\bigoplus\limits_{i=1}^n(f_i(\top_L,\cdots,\top_L)\otimes \top_L)\\
			& = & \bigoplus\limits_{i=1}^n f_i(\top_L,\cdots,\top_L)=\top_L
		\end{eqnarray*}
	}
\end{proof}	

\begin{prop}
	Given  $\langle L, \bigoplus, \otimes \rangle$ and a distributive family
	$\Gamma=\{f_i:L^n\rightarrow L:\ i=1,2,\cdots,n\}$, then $DYOWA_{\Gamma}$  satisfies {\bf (IP)}. Futhermore,
	\begin{enumerate}
		\item[(i)] If $f$ satisfies {\bf (SP)}, then $DYOWA_{\Gamma}$ also satisfies {\bf (SP)};
		\item[(ii)] $DYOWA_{\Gamma}$ does not satisfy {\bf (ZD)} and {\bf (OD)}. 
	\end{enumerate}
	
\end{prop}
\begin{proof}
	\hfill\
	\begin{enumerate}
		\item[{\bf (IP)}] Let $\Gamma$ be distributive family, then for all  $a\in L$
		{\small\begin{eqnarray*}
				DYOWA_{\Gamma}(a,a,\cdots,a) & = &\bigoplus\limits_{i=1}^n(f_i(a,a,\cdots,a)\otimes a)\\
				& = & \left(\bigoplus\limits_{i=1}^n f_i(a,a,\cdots,a)\right)\otimes a\\
				& = & a
			\end{eqnarray*}
		}
		\item[{\bf (SP)}] According to  Proposition 2.$v$, for any permutation $\sigma$ on the set $\{1,2,\cdots,n\}$, we have $\mathscr{L}\!\mathscr{M}(a_1,\cdots,a_n)=(b_1,\cdots,b_n)=\mathscr{L}\!\mathscr{M}(a_{\sigma(1)},$ $\cdots,a_{\sigma(n)})$. Therefore,
		{\scriptsize\begin{eqnarray*}
				DYOWA_{\Gamma}(a_{\sigma(1)},\cdots,a_{\sigma(n)}) & = & \bigoplus\limits_{i=1}^n (f_i(a_{\sigma(1)},\cdots,a_{\sigma(n)})\otimes b_i)\\
				& = & \bigoplus\limits_{i=1}^n (f_i(a_1,\cdots,a_n)\otimes b_i)\\
				& = & DYOWA_{\Gamma} (a_1,\cdots,a_n)
			\end{eqnarray*}
		}
		\item[{\bf (ZD)}] If $DYOWA_{\Gamma}$ have a zero divisor $a\ne\bot_L$, then\linebreak $DYOWA_{\Gamma}(a,a\cdots,a)=\bot_L$, but since $DYOWA_{\Gamma}$ always satisfies {\bf (IP)} we also have that \linebreak$DYOWA_{\Gamma}(a,\cdots,a)=a\ne \bot_L$. So, $DYOWA_{\Gamma}$ does not satisfy {\bf (ZD)}.
		
		\item[{\bf (OD)}] Analogously, if $DYOWA_{\Gamma}$ has a one divisor, $a\ne \top_L$, then $DYOWA_{\Gamma}(a,a,\cdots,a)=\top_L$, but \linebreak$DYOWA_{\Gamma}(a,\cdots,a)=a\ne \top_L$. Thus, $DYOWA_{\Gamma}$ does not satisfy {\bf (OD)}.
	\end{enumerate}
\end{proof}

\begin{prop}
	Given  $\langle L, \bigoplus, \otimes \rangle$ and a distributive family $\Gamma=\{f_i:L^n\rightarrow L:\ i=1,2,\cdots,n\}$, then
	$$\bigwedge_{i=1}^n a_i\le DYOWA_{\Gamma}(a_1,a_2,\cdots,a_n)\le \bigvee_{i=1}^n a_1$$
	for any $(a_1,a_2,\cdots,a_n)\in L^n$.
\end{prop}

\begin{proof}
	As $b_n=a_1\wedge a_2\wedge\cdots \wedge a_n\le a_i\le a_1\vee a_2\vee \cdots \vee a_n=b_1$, for all $i\in\{1,\cdots,n\}$, and t-norms and t-conorms are isotonic functions, we have
	\begin{eqnarray*}
		b_n& =& \bigoplus\limits_{i=1}^n(f_i(a_1,\cdots,a_n)\otimes b_n)\\
		& \le& \bigoplus\limits_{i=1}^n(f_i(a_1,\cdots,a_n)\otimes b_i)\\
		& \le & \bigoplus\limits_{i=1}^n(f_i(a_1,\cdots,a_n)\otimes b_1)=b_1
	\end{eqnarray*}
	Therefore,
	$$\bigwedge_{i=1}^n a_i\le DYOWA_{\Gamma}(a_1,a_1,\cdots,a_n)\le \bigvee_{i=1}^n a_1$$
\end{proof}

Theorem \ref{th:YagerIsLM} states  that OWAs are instances  of LMOWAs. the next Theorem shows that DYOWAs generalize LMOWAs.

\begin{teo}
	Let $\overrightarrow{w}=(w_1,\cdots,w_n)$ a distributive vector of weights in a complete latice $\langle L,\bigoplus,\otimes \rangle$ and $LMOWA_{\overrightarrow{w}}$ the Lizasoain-Moreno OWA, then there is a distributive family $\Gamma=\{f_i:L^n\rightarrow L:\ i=1,2,\cdots,n\}$ such that for any $\overrightarrow{a}=(a_1,\cdots,a_n)\in L^n$ we have $DYOWA_\Gamma(\overrightarrow{a})=LMOWA_{\overrightarrow{w}}(\overrightarrow{a})$, that is, DYOWA functions are provides a generalization of Lizasoain-Moreno OWA.
\end{teo} 

\begin{proof}
	It is enough to define the constant functions $f_i(\overrightarrow{a})=w_i$, since:
	\begin{eqnarray*}
			DYOWA_\Gamma(\overrightarrow{a}) & = & \bigoplus\limits_{i=1}^n \left(f_i(\overrightarrow{a})\otimes b_i\right)\\
			& = & \bigoplus\limits_{i=1}^n \left(w_i\otimes b_i\right)\\
			& = & LMOWA_{\overrightarrow{w}}(\overrightarrow{a})
	\end{eqnarray*} 
\end{proof}

\begin{cor}
	Any Yager's OWA is also a DYOWA function.
\end{cor}

\begin{proof}
	By Theorem \ref{th:YagerIsLM}, Yager's \textsf{OWA} can be written by using the T-norm $T_P$ and the T-conorm $S_{LK}$, and by Theorem 3, DYOWA's generalizes Lizasoain-Moreno functions.
\end{proof}

\vspace{-1.0cm}
\section{Conclusions and Future Works}

The \textsf{OWA} functions, of Yager, has several applications in the image processing and decision-making fields, however these operators are limited to the preset of a vector of weights from which all values of the function are calculated. In this work we define a new generalized notion of \textsf{OWA} for complete lattices environment, which goes beyond the generalization proposed by Lizasoain and Moreno, in which the weight vector is not fixed but is obtained from the input vector. We believe that this adaptive condition of the weights will fit to applications in which OWAs cannot be applied due to its limitation of fixed weight. Some tests can be found in  \citep{FariasDYOWA}.

In a future work, we intend to study the applications of the \textbf{DYOWA} functions, for example, in decision making problems, in image processing, and in other possible applications.

\section*{Compliance with ethical standards}

{\bf Conflict of interest} The authors declare that there is no conflict of interests regarding the publication of this paper.

\bibliography{refs}

\begin{thebibliography}{42}
\providecommand{\natexlab}[1]{#1}
\providecommand{\url}[1]{{#1}}
\providecommand{\urlprefix}{URL }
\expandafter\ifx\csname urlstyle\endcsname\relax
  \providecommand{\doi}[1]{DOI~\discretionary{}{}{}#1}\else
  \providecommand{\doi}{DOI~\discretionary{}{}{}\begingroup
  \urlstyle{rm}\Url}\fi
\providecommand{\eprint}[2][]{\url{#2}}

\bibitem[{Ahn(2008)}]{Ahn1}
Ahn BS (2008) Preference relation approach for obtaining {OWA} operators
  weights. International Journal of Approximate Reasoning 47(2):166 -- 178,
  \doi{http://dx.doi.org/10.1016/j.ijar.2007.04.001}

\bibitem[{Amin and Emrouznejad(2011{\natexlab{a}})}]{Amin1}
Amin GR, Emrouznejad A (2011{\natexlab{a}}) Optimizing search engines results
  using linear programming. Expert Systems with Applications 38(9):11,534 --
  11,537, \doi{http://dx.doi.org/10.1016/j.eswa.2011.03.030}

\bibitem[{Amin and Emrouznejad(2011{\natexlab{b}})}]{Amin2}
Amin GR, Emrouznejad A (2011{\natexlab{b}}) Parametric aggregation in ordered
  weighted averaging. International Journal of Approximate Reasoning 52(6):819
  -- 827, \doi{http://dx.doi.org/10.1016/j.ijar.2011.02.004}

\bibitem[{Baets and Mesiar(1999)}]{Mesiar}
Baets BD, Mesiar R (1999) Triangular norms on product lattices. Fuzzy Sets and
  Systems 104(1):61 -- 75,
  \doi{http://dx.doi.org/10.1016/S0165-0114(98)00259-0}

\bibitem[{Bedregal and Takahashi(2006)}]{BedregalTakahashi2006}
Bedregal BC, Takahashi A (2006) Interval valued versions of t-conorms, fuzzy
  negations and fuzzy implications. In: 2006 IEEE International Conference on
  Fuzzy Systems, pp 1981--1987, \doi{10.1109/FUZZY.2006.1681975}

\bibitem[{Bedregal et~al(2006)Bedregal, Santos, and
  Callejas-Bedregal}]{Bedregal2006}
Bedregal BC, Santos HS, Callejas-Bedregal R (2006) T-norms on bounded lattices:
  t-norm morphisms and operators. In: 2006 IEEE International Conference on
  Fuzzy Systems, pp 22--28, \doi{10.1109/FUZZY.2006.1681689}

\bibitem[{Beliakov et~al(2016)Beliakov, Bustince, and Calvo}]{Gleb2016}
Beliakov G, Bustince H, Calvo T (2016) A Practical Guide to Averaging
  Functions, Studies in Fuzziness and Soft Computing, vol 329. Springer,
  \doi{10.1007/978-3-319-24753-3}

\bibitem[{Birkhoff(1961)}]{Birkhoff}
Birkhoff G (1961) Lattice Theory, AMS Colloquium Publications, vol~25. American
  Mathematical Society

\bibitem[{Bustince et~al(2010)Bustince, Fernandez, Mesiar, Montero, and
  Orduna}]{overlap}
Bustince H, Fernandez J, Mesiar R, Montero J, Orduna R (2010) Overlap
  functions. Nonlinear Analysis: Theory, Methods $\&$ Applications 72(3):1488
  -- 1499, \doi{http://dx.doi.org/10.1016/j.na.2009.08.033}

\bibitem[{Bustince et~al(2011)Bustince, Paternain, Baets, Calvo, Fodor, Mesiar,
  Montero, and Pradera}]{Bustince}
Bustince H, Paternain D, Baets BD, Calvo T, Fodor J, Mesiar R, Montero J,
  Pradera A (2011) Two Methods for Image Compression/Reconstruction Using OWA
  Operators, Springer Berlin Heidelberg, Berlin, Heidelberg, pp 229--253.
  \doi{10.1007/978-3-642-17910-5$\underline{\ }$12}

\bibitem[{Bustince et~al(2012)Bustince, Pagola, Mesiar, Hullermeier, and
  Herrera}]{Bustince2012}
Bustince H, Pagola M, Mesiar R, Hullermeier E, Herrera F (2012) Grouping,
  overlap, and generalized bientropic functions for fuzzy modeling of pairwise
  comparisons. IEEE Transactions on Fuzzy Systems 20(3):405--415,
  \doi{10.1109/TFUZZ.2011.2173581}

\bibitem[{Bustince et~al(2013)Bustince, Galar, Bedregal, Koles\'arov\'a, and
  Mesiar}]{Bustince2013}
Bustince H, Galar M, Bedregal B, Koles\'arov\'a A, Mesiar R (2013) A new
  approach to interval-valued {C}hoquet integrals and the problem of ordering
  in interval-valued fuzzy set applications. IEEE Transactions on Fuzzy Systems
  21(6):1150 -- 1162, \doi{10.1109/TFUZZ.2013.2265090}

\bibitem[{Chen and Zhou(2011)}]{Chen}
Chen H, Zhou L (2011) An approach to group decision making with interval fuzzy
  preference relations based on induced generalized continuous ordered weighted
  averaging operator. Expert Systems with Applications 38(10):13,432 -- 13,440,
  \doi{http://dx.doi.org/10.1016/j.eswa.2011.04.175}

\bibitem[{Chen and Hwang(1992)}]{Chen2}
Chen SJ, Hwang CL (1992) Fuzzy Multiple Attribute Decision Making: Methods and
  Applications, Lecture Notes in Economics and Mathematical Systems, vol 375.
  Springer, Heidelberg, Berlin, \doi{10.1007/978-3-642-46768-4}

\bibitem[{Cheng and Chang(2006)}]{Cheng}
Cheng CH, Chang JR (2006) {MCDM} aggregation model using situational {ME-OWA}
  and {ME-OWGA} operators. International Journal of Uncertainty, Fuzziness and
  Knowledge-Based Systems 14(04):421--443, \doi{10.1142/S0218488506004102}

\bibitem[{Cooman and Kerre(1994)}]{DeCooman1994}
Cooman GD, Kerre E (1994) Order norms on bounded partially ordered sets. Fuzzy
  Mathematics 2:281--310

\bibitem[{Dimuro and Bedregal(2014)}]{DIMURO2014}
Dimuro GP, Bedregal B (2014) Archimedean overlap functions: The ordinal sum and
  the cancellation, idempotency and limiting properties. Fuzzy Sets and Systems
  252:39 -- 54, \doi{http://dx.doi.org/10.1016/j.fss.2014.04.008}

\bibitem[{Dubois and Prade(1985)}]{DUBOIS1985}
Dubois D, Prade H (1985) A review of fuzzy set aggregation connectives.
  Information Sciences 36(1):85 -- 121,
  \doi{http://dx.doi.org/10.1016/0020-0255(85)90027-1}

\bibitem[{Dubois and Prade(2004)}]{DUBOIS2004}
Dubois D, Prade H (2004) On the use of aggregation operations in information
  fusion processes. Fuzzy Sets and Systems 142(1):143 -- 161,
  \doi{http://dx.doi.org/10.1016/j.fss.2003.10.038}

\bibitem[{Emrouznejad(2008)}]{Emrouznejad}
Emrouznejad A (2008) {MP-OWA}: The most preferred {OWA} operator.
  Knowledge-Based Systems 21(8):847 -- 851,
  \doi{http://dx.doi.org/10.1016/j.knosys.2008.03.057}

\bibitem[{Farias et~al(2016{\natexlab{a}})Farias, Costa, Lopes, Bedregal, and
  Santiago}]{FariasDYOWA}
Farias ADS, Costa VS, Lopes LRA, Bedregal BC, Santiago RHN (2016{\natexlab{a}})
  A method of image reduction and noise reduction based on a generalization of
  ordered weighted averaging functions. arXiv:160103785

\bibitem[{Farias et~al(2016{\natexlab{b}})Farias, Lopes, Bedregal, and
  Santiago}]{FariasJIFS2016}
Farias ADS, Lopes LRA, Bedregal BC, Santiago RHN (2016{\natexlab{b}}) Closure
  properties for fuzzy recursively enumerable languages and fuzzy recursive
  languages. Journal of Intelligent $\&$ Fuzzy Systems 31(3):1795--1806,
  \doi{10.3233/JIFS-152489}

\bibitem[{Farias et~al(2016{\natexlab{c}})Farias, Santiago, and
  Bedregal}]{Farias2}
Farias ADS, Santiago RHN, Bedregal B (2016{\natexlab{c}}) Some properties of
  generalized mixture functions. In: 2016 IEEE International Conference on
  Fuzzy Systems (FUZZ-IEEE), pp 288--293, \doi{10.1109/FUZZ-IEEE.2016.7737699}

\bibitem[{Gierz et~al(1980)Gierz, Hofmann, Keimel, Lawson, Mislove, and
  Scott}]{Gierz}
Gierz G, Hofmann KH, Keimel K, Lawson JD, Mislove M, Scott DS (1980) A
  compendium of continuous lattices. Springer, \doi{10.1007/978-3-642-67678-9}

\bibitem[{Klement et~al(2000)Klement, Mesiar, and Pap}]{Klement}
Klement EP, Mesiar R, Pap E (2000) Triangular Norms, vol~8. Springer

\bibitem[{Kulisch and Miranker(1981)}]{Kulisch}
Kulisch UB, Miranker WL (1981) Computer Arithmetic Theory and Pratice. Academic
  Press, San Diego, \doi{https://doi.org/10.1016/B978-0-12-428650-4.50001-4}

\bibitem[{Liang and Xu(2014)}]{Liang2014}
Liang X, Xu W (2014) Aggregation method for motor drive systems. Electric Power
  Systems Research 117:27 -- 35,
  \doi{http://dx.doi.org/10.1016/j.epsr.2014.07.022}

\bibitem[{Lin and Jiang(2014)}]{Lin}
Lin J, Jiang Y (2014) Some hybrid weighted averaging operators and their
  application to decision making. Information Fusion 16:18 -- 28,
  \doi{http://dx.doi.org/10.1016/j.inffus.2011.06.001}

\bibitem[{Lizasoain and Moreno(2013)}]{Lizasoain}
Lizasoain I, Moreno C (2013) {OWA} operators defined on complete lattices.
  Fuzzy Sets and Systems 224:36 -- 52,
  \doi{http://dx.doi.org/10.1016/j.fss.2012.10.012}

\bibitem[{Llamazares(2015)}]{Llamazares}
Llamazares B (2015) Constructing {C}hoquet integral-based operators that
  generalize weighted means and owa operators. Information Fusion 23:131 --
  138, \doi{http://dx.doi.org/10.1016/j.inffus.2014.06.003}

\bibitem[{Merig\'o(2012)}]{Merigo2}
Merig\'o JM (2012) Probabilities in the {OWA} operator. Expert Systems with
  Applications 39(13):11,456 -- 11,467,
  \doi{http://dx.doi.org/10.1016/j.eswa.2012.04.010}

\bibitem[{Merig\'o and Gil-Lafuente(2009)}]{Merigo1}
Merig\'o JM, Gil-Lafuente AM (2009) The induced generalized {OWA} operator.
  Information Sciences 179(6):729 -- 741,
  \doi{http://dx.doi.org/10.1016/j.ins.2008.11.013}

\bibitem[{Miguel et~al(2016)Miguel, Bustince, Pekala, Bentkowska, da~Silva,
  Bedregal, Mesiar, and Ochoa}]{DeMiguel2016}
Miguel LD, Bustince H, Pekala B, Bentkowska U, da~Silva I, Bedregal B, Mesiar
  R, Ochoa G (2016) Interval-valued atanassov intuitionistic {OWA} aggregations
  using admissible linear orders and their application to decision making. IEEE
  Trans Fuzzy Systems 24(6):1586--1597

\bibitem[{Palmeira et~al(2014)Palmeira, Bedregal, Mesiar, and
  Fernandez}]{PALMEIRA2014}
Palmeira E, Bedregal B, Mesiar R, Fernandez J (2014) A new way to extend
  t-norms, t-conorms and negations. Fuzzy Sets and Systems 240:1 -- 21,
  \doi{http://dx.doi.org/10.1016/j.fss.2013.05.008}

\bibitem[{Paternain et~al(2012)Paternain, Jurio, Barrenechea, Bustince,
  Bedregal, and Szmidt}]{Paternain2012}
Paternain D, Jurio A, Barrenechea E, Bustince H, Bedregal B, Szmidt E (2012) An
  alternative to fuzzy methods in decision-making problems. Expert Systems with
  Applications 39(9):7729 -- 7735,
  \doi{http://dx.doi.org/10.1016/j.eswa.2012.01.081}

\bibitem[{Paternain et~al(2015)Paternain, Fernandez, Bustince, Mesiar, and
  Beliakov}]{Paternain2015}
Paternain D, Fernandez J, Bustince H, Mesiar R, Beliakov G (2015) Construction
  of image reduction operators using averaging aggregation functions. Fuzzy
  Sets and Systems 261:87 -- 111,
  \doi{http://dx.doi.org/10.1016/j.fss.2014.03.008}

\bibitem[{Torra and Godo(2002)}]{Torra}
Torra V, Godo L (2002) Continuous WOWA Operators with Application to
  Defuzzification, Physica-Verlag HD, Heidelberg, pp 159--176.
  \doi{10.1007/978-3-7908-1787-4$\underline{\ }$4}

\bibitem[{Yager(1988)}]{Yager1988}
Yager RR (1988) Ordered weighted averaging aggregation operators in
  multicriteria decision making. IEEE Transactions on Systems, Man, and
  Cybernetics 18:183--190

\bibitem[{Yager(2006)}]{Yager2006}
Yager RR (2006) Centered owa operators. Soft Computing 11(7):631--639,
  \doi{10.1007/s00500-006-0125-z}

\bibitem[{Zadeh(1965)}]{ZADEH1965}
Zadeh LA (1965) Fuzzy sets. Information and Control 8(3):338 -- 353,
  \doi{http://dx.doi.org/10.1016/S0019-9958(65)90241-X}

\bibitem[{Zadrozny and Kacprzyk(2006)}]{Zadrozny}
Zadrozny S, Kacprzyk J (2006) On tuning {OWA} operators in a flexible querying
  interface. In: 7th International Conference on Flexible Query Answering
  Systems, in: Lect. Notes Comput. Sci., Berlin, Heidelberg, vol 4027, pp
  97--108, \doi{10.1007/11766254$\underline{\ }$9}

\bibitem[{Zhou et~al(2008)Zhou, Chiclana, John, and Garibaldi}]{Zhou2008}
Zhou SM, Chiclana F, John RI, Garibaldi JM (2008) Type-1 owa operators for
  aggregating uncertain information with uncertain weights induced by type-2
  linguistic quantifiers. Fuzzy Sets and Systems 159(24):3281 -- 3296,
  \doi{http://dx.doi.org/10.1016/j.fss.2008.06.018}

\end{thebibliography}
\bibliographystyle{spbasic}

\end{document}